\documentclass{amsart}

\usepackage{amssymb}
\usepackage{graphicx}
\usepackage[cmtip,all]{xy}
\usepackage{microtype}
\usepackage{xcolor}
\usepackage{tikz}

\usepackage[style=alphabetic]{biblatex} 
    \DeclareFieldFormat[article,book,inbook,incollection,inproceedings,patent,thesis,unpublished]{title}{\emph{#1\isdot}}

\usepackage{hyperref}
    \hypersetup{colorlinks, linkcolor={blue!60!black}, citecolor={blue!60!black}, urlcolor={blue!60!black}}

\usepackage[nameinlink]{cleveref}
    \crefname{equation}{}{}

\calclayout

\theoremstyle{plain}
    \newtheorem{theorem}{Theorem}[section]
    \newtheorem{lemma}[theorem]{Lemma}
    
    \newtheorem{corollary}[theorem]{Corollary}
    
    \newtheorem{question}[theorem]{Question}
\theoremstyle{definition}
    \newtheorem{definition}[theorem]{Definition}

\theoremstyle{remark}
    \newtheorem{remark}[theorem]{Remark}
    \numberwithin{equation}{section}

\let\Im\relax

\def\ceil#1{\lceil #1 \rceil}
\def\et{\mathrm{\acute{e}t}}
\def\bcube{{\overline{\square}}}
\def\logrm{\mathrm{log}}
\def\RGamma{\mathrm{R}\Gamma}
\def\WOmega{\mathrm{W}\Omega}

\def\MCor{\mathrm{\underline{M}Cor}}
\def\MDM{\mathrm{\underline{M}DM}}

\def\MOmega{\mathrm{\underline{M}\Omega}}
\def\boldMOmega{\mathbf{\underline{M}\Omega}}
\def\MW{\underline{\mathrm{M}}\mathrm{W}}
\def\boldMW{\underline{\mathbf{M}}\mathbf{W}}

\def\defrm#1{\expandafter\def\csname #1\endcsname{{\mathrm {#1}}}}
\def\defop#1{\expandafter\DeclareMathOperator\csname #1\endcsname{#1}}
\def\defrms{\forcsvlist{\defrm}}
\def\defops{\forcsvlist{\defop}}

\defrms{pr,id,op,sm,prop,fin,tr,eff,div,ch,ls,red}
\defrms{Ring,Sch,Sm,Cor,DM,Ab,Zar,Nis}
\defops{Im,Ker,Coker,Sh,PSh,Fil,Pic,Tr,RHom,Ext,CDiv,Bl,Frac,Spec,trdeg,codim,cofib,map,colim,dlog}

\makeatletter
\def\quotprojlim{%
  \mathop{``\mathpalette\varlim@{\leftarrowfill@\textstyle}\hbox to 0pt{$"$}}\nmlimits@
}
\def\quotinjlim{%
  \mathop{``\mathpalette\varlim@{\rightarrowfill@\textstyle}"}\nmlimits@
}
\makeatother

\begin{document}
\title{Blow-up invariance of cohomology theories with modulus}

\author[J. Koizumi]{Junnosuke Koizumi}
\address{Graduate School of Mathematical Sciences, University of Tokyo, 3-8-1 Komaba, Meguro-ku, Tokyo 153-8914, Japan}
\email{jkoizumi@ms.u-tokyo.ac.jp}

\date{\today}
\thanks{}

\begin{abstract}
    In this paper, we study cohomology theories of $\mathbb{Q}$-modulus pairs, which are pairs $(X, D)$ consisting of a scheme $X$ and a $\mathbb{Q}$-divisor $D$.
    Our main theorem provides a sufficient condition for such a cohomology theory to be invariant under blow-ups with centers contained in the divisor.
    This yields a short proof of the blow-up invariance of the Hodge cohomology with modulus proved by Kelly-Miyazaki.
    We also define the Witt vector cohomology with modulus using the Brylinski-Kato filtration and prove its blow-up invariance.
\end{abstract}


\maketitle
\setcounter{tocdepth}{1}
\tableofcontents

\enlargethispage*{20pt}

\section{Introduction}

In the 1990s, Voevodsky developed the theory of mixed motives that unifies $\mathbb{A}^1$-invariant cohomology theories of smooth algebraic schemes over a perfect field $k$.
On the other hand, there are many cohomology theories that do not have $\mathbb{A}^1$-invariance.
Typical examples are the Hodge cohomology $\RGamma(X,\Omega^q_X)$ and the Hodge-Witt cohomology $\RGamma(X,\WOmega^q_X)$.
In order to incorporate these invariants into the theory of motives, Kahn-Miyazaki-Saito-Yamazaki \cite{KMSY1,KMSY2,KMSY3} proposed a framework called \emph{motives with modulus}.

The basic idea of motives with modulus is to replace smooth schemes by \emph{modulus pairs}.
A modulus pair is a pair $\mathcal{X}=(X,D_X)$ where $X$ is an algebraic $k$-scheme and $D_X$ is an effective Cartier divisor on $X$ such that $X^\circ:=X-|D_X|$ is smooth.
Modulus pairs and finite correspondences between them constitute an additive category $\MCor$ with a tensor structure $\mathcal{X}\otimes\mathcal{Y}=(X\times Y, \pr_1^*D_X+\pr_2^*D_Y)$.
Just as $\mathbb{A}^1$ plays the role of the interval in Voevodsky's theory, the modulus pair $\bcube:=(\mathbb{P}^1,[\infty])$, which is called the \emph{cube}, plays the role of the interval.

A Nisnevich sheaf of abelian groups $F$ on $\MCor$ which satisfies the cube invariance $F(\mathcal{X})\simeq F(\mathcal{X}\otimes\bcube)$ is simply called a \emph{cube invariant sheaf}.
Various studies have been conducted for cube invariant sheaves and their cohomolgies following analogies in Voevodsky's theory of $\mathbb{A}^1$-invariant sheaves.
The following natural conditions are often imposed on $F$:
\begin{enumerate}
	\item[(a)]	(semipurity) $F(\mathcal{X})\to F(X^\circ)$ is injective for any $\mathcal{X}\in \MCor$.
	\item[(b)]	(M-reciprocity \cite[Definition 1.26]{Sai20}) $F$ is left Kan extended from the full subcategory of $\MCor$ spanned by proper modulus pairs.
\end{enumerate}

We write $\MCor^\ls$ for the full subcategory of $\MCor$ spanned by modulus pairs $(X,D_X)$ such that $X$ is smooth and $|D_X|$ has simple normal crossings.
The following results are of particular importance:

\begin{theorem}\label{Saito_BRS}
    Let $F$ be a cube invariant sheaf which satisfies (a) and (b).
    \begin{enumerate}
        \item (Saito \cite[Theorem 9.3]{Sai20}) For any $\mathcal{X}\in \MCor^\ls$, the morphism $\RGamma(X,F_\mathcal{X})\to \RGamma(X\times\mathbb{P}^1,F_{\mathcal{X}\otimes \bcube})$ is an isomorphism.
        \item (Binda-R\"ulling-Saito \cite[Lemma 2.10]{BRS}) For any effective divisors $E_1,\dots,E_n$ on $\mathbb{A}^1$ and $\mathcal{X}\in \MCor^\ls$, we have $\mathrm{R}^i\pr_{1,*}F_{\mathcal{X}\otimes(\mathbb{A}^1,E_1)\otimes\cdots\otimes(\mathbb{A}^1,E_n)}=0$ for $i>0$.
    \end{enumerate}
\end{theorem}

In this paper, we investigate the following question posed by Kahn-Miyazaki-Saito-Yamazaki: 

\begin{question}[Kahn-Miyazaki-Saito-Yamazaki {\cite[Question 2]{KMSY1}}]\label{KMSY_question}
    Let $F$ be a cube invariant sheaf which satisfies (a) and (b).
    Let $\mathcal{X}\in \MCor^\ls$ and let $Z\subset |D_X|$ be a smooth closed subscheme which has simple normal crossings with $|D_X|$.
    Is the morphism
    $$
        \RGamma(X,F_\mathcal{X})\to \RGamma(\Bl_ZX,F_{\Bl_Z\mathcal{X}})
    $$
    an isomorphism?
\end{question}

Binda-R\"ulling-Saito proved that \Cref{KMSY_question} is true if $D_X$ is reduced \cite[Theorem 2.12]{BRS}.
Recently, Kelly-Miyazaki \cite{KellyMiyazaki_Hodge1,KellyMiyazaki_Hodge2} defined the sheaf of differential forms $\MOmega^q$ for modulus pairs over a field of characteristic $0$, which is characterized by the formula
$$
    \MOmega^q(\mathcal{X})=\Gamma\bigl(X,\Omega^q(\log |D_X|)(D_X-|D_X|)\bigr)
$$
for $\mathcal{X}\in \MCor^\ls$, and proved that \Cref{KMSY_question} is true for $F=\MOmega^q$ \cite[Theorem 6.1]{KellyMiyazaki_Hodge2}.
On the other hand, R\"ulling-Saito showed that \Cref{KMSY_question} is not true in general \cite[Section 6.9]{RSram3}.

In this paper, we take a new approach to this problem: we use \emph{$\mathbb{Q}$-modulus pairs} instead of modulus pairs, which allows $D_X$ to be a \emph{$\mathbb{Q}$-divisor}.
This concept was introduced by the author and Miyazaki \cite{Koizumi-Miyazaki} to give a motivic construction of the de Rham-Witt complex.
There is a category of $\mathbb{Q}$-modulus pairs $\MCor^\mathbb{Q}$ which is similar to $\MCor$.
We consider the following natural conditions:

\begin{definition}
    Let $F$ be a Nisnevich sheaf on $\MCor^\mathbb{Q}$.
    \begin{enumerate}
        \item We say that $F$ has \emph{cohomological cube invariance} if for any $\mathcal{X}\in \MCor^{\mathbb{Q},\ls}$, the morphism $\RGamma(X,F_\mathcal{X})\to \RGamma(X\times\mathbb{P}^1,F_{\mathcal{X}\otimes \bcube})$ is an isomorphism.
        \item We say that $F$ has \emph{affine vanishing property} if for any effective divisors $E_1,\dots,E_n$ on $\mathbb{A}^1$ and $\mathcal{X}\in \MCor^{\mathbb{Q},\ls}$, we have $\mathrm{R}^i\pr_{1,*}F_{\mathcal{X}\otimes(\mathbb{A}^1,E_1)\otimes\cdots\otimes(\mathbb{A}^1,E_n)}=0$ for $i>0$.
        \item We say that $F$ has \emph{left continuity} if the canonical map
    $\colim_{\varepsilon \to 0} F\bigl(X,(1-\varepsilon)D_X\bigr)\to F(\mathcal{X})$
    is an isomorphism for any $\mathcal{X}\in \MCor^{\mathbb{Q},\ls}$.
    \end{enumerate}
\end{definition}

The main result of this paper is the following:

\begin{theorem}\label{main}
    Let $F$ be a Nisnevich sheaf on $\MCor^\mathbb{Q}$ which has cohomological cube invariance, affine vanishing property, and left continuity.
    Let $\mathcal{X}\in \MCor^\ls$ and let $Z\subset |D_X|$ be a smooth closed subscheme which has simple normal crossings with $|D_X|$.
    Then the morphism
    $$
        \RGamma(X,F_\mathcal{X})\to \RGamma(\Bl_ZX,F_{\Bl_Z\mathcal{X}})
    $$
    is an isomorphism.
\end{theorem}

Our proof of \Cref{main} is inspired by the proof of \cite[Theorem 2.12]{BRS}.
We reduce to the case of the blow-up of $\mathbb{A}^2$ at the origin, and then relate it to some other toric surfaces to deduce the required vanishing result from the cohomological cube invariance and the affine vanishing property.
Note that the counterexample to \Cref{KMSY_question} given by R\"ulling-Saito \cite[Section 6.9]{RSram3} shows that the left continuity assumption cannot be dropped.

We apply \Cref{main} to two sheaves.
Firstly, we extend the sheaf $\MOmega^q$ defined by Kelly-Miyazaki to $\mathbb{Q}$-modulus pairs over an arbitrary field, and show that it satisfies the assumption of the theorem.
This yields a short proof of the blow-up invariance of the Hodge cohomology with modulus due to Kelly-Miyazaki \cite[Theorem 6.1]{KellyMiyazaki_Hodge2}.
Secondly, assuming $\ch(k)=p>0$, we construct a sheaf $\MW_n$ on $\MCor^\mathbb{Q}$ which extends the sheaf of Witt vectors.
For $\mathcal{X}\in \MCor^\ls$, the group $\MW_n(\mathcal{X})$ is given by
$$
	\MW_n(\mathcal{X})=\Gamma\bigl(X,\mathrm{W}_n\mathcal{O}_X((D_X-|D_X|)/p^{n-1})\bigr).
$$
Here, $\mathrm{W}_n\mathcal{O}_X(D)$ denotes the \emph{Witt divisorial sheaf} introduced by Tanaka \cite{Tanaka_Witt}, which is defined by using the Brylinski-Kato filtration on the ring of Witt vectors.
We check that $\MW_n$ also satisfies the assumption of the theorem, and hence obtain the following result:
\begin{corollary}\label{Witt_CBI}
    Let $\mathcal{X}\in \MCor^{\mathbb{Q},\ls}$ and let $Z\subset |D_X|$ be a smooth closed subscheme which has simple normal crossings with $|D_X|$.
    Then the morphism
    $$
        \RGamma\bigl(X,(\MW_n)_\mathcal{X}\bigr)\to \RGamma\bigl(\Bl_ZX,(\MW_n)_{\Bl_Z\mathcal{X}}\bigr)
    $$
    is an isomorphism.
\end{corollary}

If we assume resolution of singularities, these results imply that these cohomology theories are representable in the category $\MDM^\eff$ of motives with modulus defined by Kahn-Miyazaki-Saito-Yamazaki \cite{KMSY3}:

\begin{corollary}\label{cor:realization}
	Assume that $k$ admits resolution of singularities.
	\begin{enumerate}
		\item	(\cite[Theorem 1.3]{KellyMiyazaki_Hodge2}) There is an object $\boldMOmega^q$ in $\MDM^\eff$ such that for any $\mathcal{X}\in \MCor^\ls$ we have an equivalence
				$$
					\map_{\MDM^\eff}\bigl(\mathrm{M}(\mathcal{X}),\boldMOmega^q\bigr) \simeq \mathrm{R}\Gamma(X,\MOmega^q_\mathcal{X}).
				$$
		\item	Suppose that $\ch(k)=p>0$.
				There is an object $\boldMW_n$ in $\MDM^\eff$ such that for any $\mathcal{X}\in \MCor^\ls$ we have an equivalence
				$$
					\map_{\MDM^\eff}\bigl(\mathrm{M}(\mathcal{X}),\boldMW_n\bigr) \simeq \mathrm{R}\Gamma\bigl(X,(\MW_n)_\mathcal{X}\bigr).
				$$
	\end{enumerate}
\end{corollary}

\subsection*{Acknowledgments}
The author would like to thank Shane Kelly, Hiroyasu Miyazaki, and Shuji Saito for helpful comments on a draft of this paper.
The author is supported by JSPS KAKENHI Grant (22J20698).

\subsection*{Conventions}
Throughout this paper, we fix a perfect field $k$.
An \emph{algebraic $k$-scheme} is a separated $k$-scheme of finite type.
For an algebraic $k$-scheme $X$, we write $X^N$ for its normalization.
We write $\Sm$ for the category of smooth algebraic $k$-schemes.
We say that a morphism of schemes $f\colon X\to Y$ is \emph{pseudo-dominant} if it sends generic points of $X$ to generic points of $Y$.
If $X$ is a smooth algebraic $k$-scheme, then a \emph{coordinate} of $X$ is a family of functions $x_1,\dots,x_d\in \Gamma(X,\mathcal{O}_X)$ which defines an \'etale morphism $X\to \mathbb{A}^d$.
Unless otherwise specified, cohomologies are taken in the Nisnevich topology.
For a field $K$, we write $K\{t_1,\dots,t_d\}$ for the henselization of $K[t_1,\dots,t_d]$ at $(t_1,\dots,t_d)$.

\section{Preliminaries}

In this section, we recall several definitions and results concerning modulus pairs.

\subsection{Finite correspondences}

First we recall Suslin-Voevodsky's category of finite correspondences $\Cor$.
It is an additive category having the same objects as $\Sm$, and its group of morphisms $\Cor(X,Y)$ is the group of algebraic cycles on $X\times Y$ whose components are finite pseudo-dominant over $X$.
The category $\Cor$ has a symmetric monoidal structure given by the fiber product of $k$-schemes.
For any morphism $f\colon X\to Y$ in $\Sm$, the graph of $f$ gives a morphism $f\colon X\to Y$ in $\Cor$.
If $f$ is finite pseudo-dominant, then the transpose of the graph of $f$ gives a morphism ${}^tf\colon Y\to X$ in $\Cor$.

We say that a presheaf $F\colon \Cor^\op\to \Ab$ is a \emph{Nisnevich sheaf} if for any $X\in \Sm$, the presheaf
$$
    F_X\colon X_\et^\op\to \Ab;\quad U\mapsto F(U)
$$
is a Nisnevich sheaf.
We write $\Sh_\Nis(\Cor)$ for the category of Nisnevich sheaves on $\Cor$.
If $f\colon X\to Y$ is a finite pseudo-dominant morphism in $\Sm$ and $F\in \Sh_\Nis(\Cor)$, then we write $f_*$ or $\Tr_{X/Y}$ for the map $({}^tf)^*\colon F(X)\to F(Y)$.

\subsection{Modulus pairs}

Next we recall basic facts about modulus pairs \cite{KMSY1}.
Since we want to allow $\mathbb{Q}$-divisors, we follow \cite{Koizumi-Miyazaki}.
We fix $\Lambda\in \{\mathbb{Z},\mathbb{Q}\}$.

Let $X$ be a notherian scheme.
A \emph{$\Lambda$-Cartier divisor} on $X$ is an element of $\Lambda\otimes_\mathbb{Z}\CDiv(X)$ where $\CDiv(X)$ is the group of Cartier divisor on $X$.
Note that a $\mathbb{Z}$-Cartier divisor is a usual Cartier divisor.
A $\Lambda$-Cartier divisor $D$ is called \emph{effective} if it is contained in the $\Lambda_{>0}$-span of some effective Cartier divisor. 
For two $\Lambda$-Cartier divisors $D_1$ and $D_2$, we write $D_1\geq D_2$ if $D_1-D_2$ is effective.
The support $|D|$ if $D$ is defined to be the support of any $E\in \CDiv(X)$ whose $\Lambda_{>0}$-span contains $D$.

A $\Lambda$-modulus pair is a pair $\mathcal{X}=(X,D_X)$ where $X$ is an algebraic $k$-scheme and $D_X$ is an effective $\Lambda$-Cartier divisor on $X$ such that $X^\circ:=X-|D_X|$ is smooth.
For example, the $\Lambda$-modulus pair $\bcube:=(\mathbb{P}^1,[\infty])$ is called the \emph{cube}.
An \emph{ambient morphism} of $\Lambda$-modulus pairs $f\colon \mathcal{X}\to \mathcal{Y}$ is a morphism of $k$-schemes $f\colon X\to Y$ such that $f(X^\circ)\subset Y^\circ$ and $D_X\geq f^*D_Y$ hold.
It is called \emph{minimal} if $D_X=f^*D_Y$ holds.
We write $f^\circ\colon X^\circ\to Y^\circ$ for the induced morphism of $k$-schemes.

Let $\mathcal{X},\mathcal{Y}$ be $\Lambda$-modulus pairs over $k$ and let $V\subset X^\circ\times Y^\circ$ be an integral closed subscheme.
We say that $V$ is \emph{left proper} if the closure $\overline{V}$ of $V$ in $X\times Y$ is proper over $X$.
We say that $V$ is \emph{admissible} if $(\pr_1^*D_X)|_{\overline{V}^N}\geq (\pr_2^*D_Y)|_{\overline{V}^N}$ holds.
We write $\MCor^\Lambda(\mathcal{X},\mathcal{Y})$ for the subgroup of $\Cor(X^\circ,Y^\circ)$ consisting of cycles whose components are left proper and admissible.
This defines a category $\MCor^\Lambda$ of $\Lambda$-modulus pairs.
For any ambient morphism $f\colon \mathcal{X}\to \mathcal{Y}$, the graph of $f^\circ$ gives a morphism $f\colon \mathcal{X}\to \mathcal{Y}$ in $\MCor^\Lambda$.
If $f$ is a proper minimal ambient morphism such that $f^\circ$ is finite pseudo-dominant, then the transpose of the graph of $f^\circ$ gives a morphism ${}^tf\colon \mathcal{Y}\to \mathcal{X}$ in $\MCor^\Lambda$.
The category $\MCor^\Lambda$ has a symmetric monoidal structure $\otimes$ given by
$$
	\mathcal{X}\otimes \mathcal{Y}=(X\times Y, \pr_1^*D_X+\pr_2^*D_Y).
$$
We can regard $\MCor:=\MCor^\mathbb{Z}$ as a full subcategory of $\MCor^\mathbb{Q}$.
We write $\MCor^{\Lambda,\ls}$ for the full subcategory of $\MCor^\Lambda$ spanned by \emph{log-smooth $\Lambda$-modulus pairs}, i.e., modulus pairs $(X,D_X)$ such that $X$ is smooth and $|D_X|$ has simple normal crossings.

\begin{definition}\label{}
    Let $\mathcal{X}\in \MCor^{\Lambda,\ls}$.
    We say that a smooth closed subscheme $Z\subset |D_X|$ has \emph{simple normal crossings with $|D_X|$} if Zariski locally on $X$, there is a coordinate $x_1,\dots,x_d$ such that we can write $|D_X|=\{x_1x_2\cdots x_s=0\}$, $Z=\{x_{i_1}=\cdots=x_{i_r}=0\}$.
\end{definition}

Let $\mathcal{X}\in \MCor^{\Lambda,\ls}$ and let $Z\subset |D_X|$ be a smooth closed subscheme which has simple normal crossings with $|D_X|$.
Let $\varphi\colon \Bl_ZX\to X$ denote the blow-up along $Z$.
In this situation we write $\Bl_Z\mathcal{X}=(\Bl_ZX,\varphi^*D_X)$.
Note that the ambient morphism $\varphi\colon \Bl_Z\mathcal{X}\to \mathcal{X}$ is an isomorphism in $\MCor^{\Lambda}$ because ${}^t\varphi$ gives an inverse.

We say that a presheaf $F\colon (\MCor^\Lambda)^\op\to \Ab$ is a \emph{Nisnevich sheaf} if for any $\mathcal{X}\in \MCor^\Lambda$, the presheaf
$$
    F_\mathcal{X}\colon X_\et^\op\to \Ab;\quad U\mapsto F(U,D_X|_U)
$$
is a Nisnevich sheaf.
We write $\Sh_\Nis(\MCor^\Lambda)$ for the category of Nisnevich sheaves on $\MCor^\Lambda$.

\begin{definition}
    Let $F$ be a Nisnevich sheaf on $\MCor^\Lambda$.
    \begin{enumerate}
        \item We say that $F$ has \emph{cohomological cube invariance} if for any $\mathcal{X}\in \MCor^{\Lambda,\ls}$, the morphism $\RGamma(X,F_\mathcal{X})\to \RGamma(X\times\mathbb{P}^1,F_{\mathcal{X}\otimes \bcube})$ is an isomorphism.
        \item We say that $F$ has \emph{affine vanishing property} if for any effective $\Lambda$-Cartier divisors $E_1,\dots,E_n$ on $\mathbb{A}^1$ and $\mathcal{X}\in \MCor^{\Lambda,\ls}$, we have $\mathrm{R}^i\pr_{1,*}F_{\mathcal{X}\otimes(\mathbb{A}^1,E_1)\otimes\cdots\otimes(\mathbb{A}^1,E_n)}=0$ for $i>0$.
        \item Suppose that $\Lambda=\mathbb{Q}$.
        We say that $F$ has \emph{left continuity} if the canonical map
    $\colim_{\varepsilon \to 0} F\bigl(X,(1-\varepsilon)D_X\bigr)\to F(\mathcal{X})$
    is an isomorphism for any $\mathcal{X}\in \MCor^{\mathbb{Q},\ls}$.
    \end{enumerate}
\end{definition}

\subsection{Ramification filtrations}
In this subsection we present a method to construct Nisnevich sheaves on $\MCor^\mathbb{Q}$ from \emph{ramification filtrations}.
This is essentially due to R\"ulling-Saito \cite{RS21}, but we make a small modification to allow $\mathbb{Q}$-divisors.

A \emph{geometric henselian DVF} is a discrete valuation field $(L,v_L)$ which is isomorphic to $\Frac \mathcal{O}^h_{X,x}$ for some $X\in \Sm$ and some point $x\in X$ of codimension $1$.
We write $\Phi$ for the class of all geometric henselian DVFs.
For $L\in \Phi$, we write $\mathcal{O}_L$ for the valuation ring and $\mathfrak{m}_L$ for the maximal ideal of $\mathcal{O}_L$.
Note that if $L\in \Phi$ and $K$ is the residue field of $\mathcal{O}_L$, then we have $\mathcal{O}_L\simeq K\{t\}$ for any uniformizer $t\in \mathcal{O}_L$.

\begin{definition}\label{def:ram_fil}
	Let $F\in \Sh_\Nis(\Cor)$.
	A \emph{ramification filtration} $\Fil$ on $F$ is a collection of increasing filtrations $\{\Fil_r F(L)\}_{r\in \mathbb{Q}_{\geq 0}}$ on $F(L)$ indexed by $L\in \Phi$ which satisfies the following:
	\begin{enumerate}
		\item	For any $L\in \Phi$, we have $\Im\bigl(F(\mathcal{O}_L)\to F(L)\bigr)\subset \Fil_0F(L)$.
		\item	If $L\in \Phi$ and $L'/L$ is a finite extension with ramification index $e$, then we have
				$$
					\Tr_{L'/L}\bigl(\Fil_r F(L')\bigr)\subset \Fil_{r/e}F(L)\quad (r\in \mathbb{Q}_{\geq 0}).
				$$
	\end{enumerate}
\end{definition}

\begin{definition}
	Let $F\in \Sh_\Nis(\Cor)$ and let $\Fil$ be a ramification filtration on $F$.
	Let $\mathcal{X}\in \MCor$ and $a\in F(X^\circ)$.
	We say that $a$ is \emph{bounded by $D_X$} if for any $L\in \Phi$ and any commutative diagram of the following form, we have $\rho^*a\in \Fil_{v_L(\widetilde{\rho}^*D_X)}F(L)$:
        \begin{align}\label{DVR_diagram}
            \xymatrix{
                \Spec L\ar[r]^-{\rho}\ar@{^(->}[d]             &X^\circ\ar@{^(->}[d]\\
                \Spec \mathcal{O}_L\ar[r]^-{\widetilde{\rho}}       &X.
            }
        \end{align}
	We write $F_{\Fil}(\mathcal{X})$ for the subgroup of $F(X^\circ)$ consisting of elements bounded by $D_X$.
\end{definition}

\begin{lemma}[cf.\,{\cite[Proposition 4.7]{RS21}}]\label{lem:ram_fil}
	Let $F\in \Sh_\Nis(\Cor)$ and let $\Fil$ be a ramification filtration on $F$.
	\begin{enumerate}
		\item	For any $\mathcal{X},\mathcal{Y}\in \MCor^{\mathbb{Q}}$ and $\alpha\in \MCor^{\mathbb{Q}}(\mathcal{X},\mathcal{Y})$, we have $\alpha^*F_{\Fil} (\mathcal{Y})\subset F_{\Fil} (\mathcal{X})$.
		\item	$F_{\Fil}$ is a Nisnevich sheaf on $\MCor^\mathbb{Q}$.
	\end{enumerate}
\end{lemma}

\begin{proof}
    \begin{enumerate}
        \item We may assume that $\alpha = [V]$ for some integral closed subscheme $V\subset X^\circ\times Y^\circ$.
        Let $b\in F_{\Fil}(\mathcal{Y})$ and $L\in \Phi$.
        Suppose that we have a commutative diagram of the form \Cref{DVR_diagram}.
        It suffices to show that $\rho^*\alpha^*b\in \Fil_{v_L(\widetilde{\rho}^*D_X)}F(L)$.

        Let us consider the composition $\Spec L\to X^\circ\to Y^\circ$.
        We can write $\rho^*[V]=\sum_{i=1}^m n_i[x_i]$ as a relative $0$-cycle on $Y^\circ\times_X\Spec L$ over $\Spec L$.
        Since $L_i:=k(x_i)$ is finite over $L$, we have $L_i\in \Phi$.
        Write $\eta_i\colon \Spec L_i\to Y$ and $\pi_i\colon \Spec L_i\to \Spec L$ for the canonical morphisms.
        Then the composition $\Spec L\to X^\circ\to Y^\circ$ is equal to $\sum_{i=1}^m n_i(\eta_i\circ {}^t\pi_i)$, so we have $\rho^*\alpha^*b=\sum_{i=1}^mn_i(\pi_{i,*}\eta_i^*b)$.
        Therefore it suffices to show that $\pi_{i,*}\eta_i^*b$ is contained in $\Fil_{v_L(\widetilde{\rho}^*D_X)}F(L)$.

        Since $\overline{V}^N$ is proper over $X$, the canonical morphism $\Spec L_i\to V^N$ uniquely extends to $\Spec \mathcal{O}_{L_i} \to \overline{V}^N$.
        We write $\widetilde{\pi}_i \colon \Spec \mathcal{O}_{L_i}\to \Spec \mathcal{O}_L$ and $\widetilde{\eta}_i\colon \Spec \mathcal{O}_{L_i}\to Y$ for the induced morphisms.
        Pulling back the inequality $(\pr_1^*D_X)|_{\overline{V}^N}\geq (\pr_2^*D_Y)|_{\overline{V}^N}$ to $\Spec \mathcal{O}_{L_i}$, we get $\widetilde{\pi}_i^*\widetilde{\rho}^*D_X\geq \widetilde{\eta}_i^*D_Y$.
        Consider the following commutative diagram:
        $$
        \xymatrix{
            \Spec L_i\ar[r]^-{\eta_i}\ar@{^(->}[d]             &Y^\circ\ar@{^(->}[d]\\
            \Spec \mathcal{O}_{L_i}\ar[r]^-{\widetilde{\eta}_i}       &Y.
        }
        $$
        Since $b\in F_{\Fil}(\mathcal{Y})$, we have $\eta_i^*b\in \Fil_{v_{L_i}(\widetilde{\eta}_i^*D_Y)}F(L_i)$.
        The above inequality shows that $\eta_i^*b\in \Fil_{v_{L_i}(\widetilde{\pi}_i^*\widetilde{\rho}^*D_X)}F(L_i)$.
        If $L_i/L$ has ramification index $e_i$, then we have $v_{L_i}(\pi_i^*\widetilde{\rho}^*D_X)=e_i\cdot v_L(\widetilde{\rho}^*D_X)$.
        By the definition of a ramification filtration, we get $\pi_{i,*}\eta_i^*b\in \Fil_{v_L(\widetilde{\rho}^*D_X)}F(L)$, as desired.
        \item Since $F_{\Fil}\subset F$, it is clear that $F_{\Fil}$ is separated for the Nisnevich topology.
        Let $\mathcal{X}\in \MCor^{\mathbb{Q}}$ and let $\{\pi_i\colon U_i\to X\}_{i\in I}$ be a Nisnevich covering of $X$.
        It suffices to show that if $a\in F(X^\circ)$ satisfies $a|_{U_i}\in F_{\Fil}(U_i,D_X|_{U_i})$ for all $i\in I$, then $a\in F_{\Fil}(\mathcal{X})$.
        Suppose that we have a commutative diagram of the form \Cref{DVR_diagram}.
        Since $L$ is henselian, $\widetilde{\rho}$ lifts to $\widetilde{\rho}'\colon \Spec \mathcal{O}_L\to U_i$ for some $i$.
        By $a|_{U_i}\in F_{\Fil}(U_i,D_X|_{U_i})$, we get $\rho^*a\in \Fil_{v_L(\widetilde{\rho}^*D_X)}F(L)$, as desired.\qedhere
    \end{enumerate}
\end{proof}

\section{Blow-up invariance}\label{section_main}

In this section we prove \Cref{main}.
First we prepare some lemmas.

\begin{lemma}\label{P1_vanishing_lemma}
	Let $A$ be a henselian local ring.
	Let $\varphi\colon F\to G$ be a morphism of Nisnevich sheaves on $\mathbb{P}^1_A$ such that $\Ker \varphi$ and $\Coker \varphi$ are supported on $\{0,\infty\}\subset\mathbb{P}^1_A$.
	Then the morphism
	$$
		\mathrm{H}^i(\mathbb{P}^1_A, F)\to \mathrm{H}^i(\mathbb{P}^1_A, G)
	$$
	is surjective for $i>0$.
\end{lemma}

\begin{proof}
	We have two exact sequences of Nisnevich sheaves on $\mathbb{P}^1_A$:
	\begin{align*}
		&0\to \Ker f\to F\xrightarrow{\varphi} \Im f\to 0,\\
		&0\to \Im f\to G\to \Coker f\to 0.
	\end{align*}
	Taking the cohomologies, we get exact sequences
	\begin{align*}
		&\mathrm{H}^i(\mathbb{P}^1_A,F)\xrightarrow{\varphi} \mathrm{H}^i(\mathbb{P}^1_A,\Im f)\to \mathrm{H}^{i+1}(\mathbb{P}^1_A,\Ker f),\\
		&\mathrm{H}^i(\mathbb{P}^1_A,\Im f)\to \mathrm{H}^i(\mathbb{P}^1_A,G)\to \mathrm{H}^i(\mathbb{P}^1_A,\Coker f).
	\end{align*}
	By our assumption, the rightmost terms are $0$ for $i>0$, so the composition
	$$
		\mathrm{H}^i(\mathbb{P}^1_A,F)\xrightarrow{\varphi} \mathrm{H}^i(\mathbb{P}^1_A,\Im f)\to \mathrm{H}^i(\mathbb{P}^1_A,G)
	$$
	is surjective for $i>0$.
\end{proof}

\begin{lemma}\label{construction_of_m}
    Let $a,b\in \mathbb{Q}_{\geq 0}$ and suppose that $a\neq 0$.
    Let $N$ be a positive integer such that $c:=a+b-\frac{a+b+1}{N+1}$ and $Na-1$ are positive.
    Then there are integers $m,m'\geq 0$ such that the following conditions hold:
    $$
        m+m'=N,\quad mc \leq Na-1,\quad m'c\leq Nb.
    $$
\end{lemma}

\begin{proof}
    By definition, we have $(N+1)c = Na+Nb-1$.
    Write $Na-1=mc+r$ and $Nb = m'c+r'$ with $0 < r \leq c$ and $0 \leq r' < c$.
    Then we get $(N+1)c = (m+m')c+(r+r')$, so $r+r'$ is a multiple of $c$.
    Since $0< r+r'< 2c$, we must have $r+r'=c$.
    Therefore we get $m+m'=N$, $mc\leq Na-1$, and $m'c\leq Nb$.
\end{proof}

The blow-up of $\mathbb{A}^2$ at the origin is isomorphic to the total space of the line bundle $\mathcal{O}_{\mathbb{P}^1}(-1)$ on $\mathbb{P}^1$.
We first study the cohomology of this scheme.
Let $H^{(n)}$ denote the total space of the line bundle $\mathcal{O}_{\mathbb{P}^1}(-n)$ on $\mathbb{P}^1$;
it is obtained by gluing $\Spec k[s,x]$ and $\Spec k[s^{-1},s^nx]$ along $\Spec k[s^\pm,x]$.
Let $\pi\colon H^{(n)}\to \mathbb{P}^1$ be the canonical projection.
We define divisors $D_0, D_\infty, E$ on $H^{(n)}$ by
$$
	D_0 = \pi^{-1}(0),\quad D_\infty = \pi^{-1}(\infty), \quad E = \text{(image of the zero section)}.
$$
We define $\mathcal{H}^{(n)}(a,b,c)=(H^{(n)}, a D_0+b D_\infty+c E)$ for $a,b,c\in \mathbb{Q}_{\geq 0}$.
In terms of toric geometry, $H^{(n)}$ corresponds to the fan $\Delta_n$ given by the faces of the cones $\sigma,\tau\subset \mathbb{R}^2$ where
\begin{align*}
    \sigma &{}= \{(x,y)\in \mathbb{R}^2\mid x,y\geq 0\},\\
    \tau &{}= \{(x,y)\in \mathbb{R}^2\mid x\leq 0,\;y\geq -nx\}.
\end{align*}
The divisors $D_0,D_\infty,E$ correspond to the rays generated by $(1,0),(-1,n),(0,1)$ respectively.

\begin{figure}[h]
    \begin{tikzpicture}[scale=1.0]
        \fill [gray!30] (0,0) -- (0,2) -- (-1,2) -- (0,0);
        \fill [gray!30] (0,0) -- (2,0) -- (2,2) -- (0,2) -- (0,0);
        \draw (-2,0) -- (2,0);
        \draw (0,-1) -- (0,2);
        \draw [thick] (0,0) -- (-1,2);
        \draw [thick] (0,2) -- (0,0) -- (2,0);
        \node at (0,2) [above] {$E$};
        \node at (2,0) [right] {$D_0$};
        \node at (-1,2) [above] {$D_\infty$};
        \node at (1,1) {$\sigma$};
        \node at (-0.4,1.5) {$\tau$};
    \end{tikzpicture}
    \caption{The fan $\Delta_n$}
\end{figure}

For $\mathcal{X}\in \MCor^{\mathbb{Q},\ls}$, we write $\mathcal{H}^{(n)}_\mathcal{X}(a,b,c) = \mathcal{H}^{(n)}(a,b,c)\otimes \mathcal{X}$ and write $\pi_X\colon H^{(n)}_X\to \mathbb{P}^1_X$ for the base change of $\pi\colon H^{(n)}\to \mathbb{P}^1$.

\begin{lemma}\label{Hirzebruch_vanishing}
    Suppose that $F\in \Sh_\Nis(\MCor^{\mathbb{Q}})$ has cohomological cube invariance and affine vanishing property.
    Let $\mathcal{X}$ be the henselian localization of an object of $\MCor^{\mathbb{Q},\ls}$ and $a,b\in \mathbb{Q}_{\geq 0}$, $a\neq 0$.
    Let $N$ be a positive integer such that $c:=a+b-\frac{a+b+1}{N+1}$ and $Na-1$ are positive.
    Then the cohomology group
    $$
        \mathrm{H}^i(H^{(1)}_X, F_{\mathcal{H}^{(1)}_\mathcal{X}(a,b,c)})
    $$
    is annihilated by $N$ for $i>0$.
\end{lemma}

\begin{proof}
    Let $\rho_{N}\colon \mathbb{P}^1_X\to \mathbb{P}^1_X$ be the morphism $s\mapsto s^N$.
    The pullback of $\mathcal{O}_{\mathbb{P}^1_X}(-1)$ by $\rho_N$ is canonically isomorphic to $\mathcal{O}_{\mathbb{P}^1_X}(-N)$, so we have the following Cartesian diagram:
    $$
    \xymatrix{
        H^{(N)}_X\ar[r]^-{\theta_N}\ar[d]		&H^{(1)}_X\ar[d]\\
        \mathbb{P}^1_X\ar[r]^-{\rho_N}				&\mathbb{P}^1_X
    }
    $$
    In terms of toric geometry, the morphism $\theta_N$ corresponds to the map of fans $\Delta_N\to \Delta_1$ induced by the map of lattices
        $$
            \begin{pmatrix}N&0\\0&1\end{pmatrix}\colon \mathbb{Z}^2\to \mathbb{Z}^2.
        $$
    The morphism $\theta_N$ is finite locally free of degree $N$, and it induces a minimal ambient morphism $\mathcal{H}_\mathcal{X}^{(N)}(Na,Nb,c)\to \mathcal{H}_\mathcal{X}^{(1)}(a,b,c)$.
    The composition
    $$
        F_{\mathcal{H}^{(1)}_\mathcal{X}(a,b,c)}\xrightarrow{\theta_N^*} \theta_{N,*} F_{\mathcal{H}^{(N)}_\mathcal{X}(Na,Nb,c)}\xrightarrow{\theta_{N,*}} F_{\mathcal{H}^{(1)}_\mathcal{X}(a,b,c)}
    $$
    is equal to the multiplication by $N$, so it suffices to show that 
    $$
        \mathrm{H}^i\bigl(H^{(N)}_X, \theta_{N,*}F_{\mathcal{H}^{(N)}_\mathcal{X}(Na,Nb,c)}\bigr)=0\quad (i>0).
    $$
    Since $\theta_N$ is finite, we have $\mathrm{R}^i\theta_{N,*}=0\;(i>0)$.
    Therefore the above group is isomorphic to
    $$
        \mathrm{H}^i\bigl(H^{(N)}_X, F_{\mathcal{H}^{(N)}_\mathcal{X}(Na,Nb,c)}\bigr).
    $$
    Also, we have $\mathrm{R}^i\pi_{X,*}F_{\mathcal{H}^{(N)}_\mathcal{X}(Na,Nb,c)}=0$ for $i>0$ by the affine vanishing property of $F$, so it suffices to prove that $\mathrm{H}^i\bigl(\mathbb{P}^1_X, \pi_{X,*}F_{\mathcal{H}_\mathcal{X}^{(N)}(Na,Nb,c)}\bigr)=0$ for $i>0$.
        
        By \Cref{construction_of_m}, there are integers $m,m'\geq 0$ such that
	$$
		m+m'=N,\quad mc \leq Na-1,\quad m'c\leq Nb.
	$$
        Let $\varphi_{m,m'}$ denote the section $\mathcal{O}_{\mathbb{P}^1_X}\to \mathcal{O}_{\mathbb{P}^1_X}(N)$ defined by $S^mT^{m'}$, where $S,T$ are the homogeneous coordinates of $\mathbb{P}^1_X$.
        Twisting by $-N$, we get a morphism $\mathcal{O}_{\mathbb{P}^1_X}(-N)\to \mathcal{O}_{\mathbb{P}^1_X}$.
	Let $\psi_{m,m'}\colon H^{(N)}_X\to H^{(0)}_X$ denote the induced morphism of schemes over $\mathbb{P}^1_X$.
    In coordinates, this can be written as $x\mapsto s^mx$.
    In terms of toric geometry, this corresponds to the map of fans $\Delta_N\to \Delta_0$ induced by the map of lattices
    $$
        \begin{pmatrix}1&0\\m&1\end{pmatrix}\colon \mathbb{Z}^2\to \mathbb{Z}^2.
    $$
    The inequalities $mc \leq Na-1$ and $m'c\leq Nb$ imply that we have
    \begin{align*}
        \psi_{m,m'}^*(D_0+cE)& = D_0 + c (mD_0 + m'D_\infty + E)\\
        &=(mc+1)D_0+m'cD_\infty + cE\\
        &\leq NaD_0+NbD_\infty+cE
    \end{align*}
    on $H^{(N)}_X$.
    Therefore $\psi_{m,m'}$ induces an ambient morphism
    $$
        \psi_{m,m'}\colon \mathcal{H}_\mathcal{X}^{(N)}(Na,Nb,c)\to \mathcal{H}^{(0)}_\mathcal{X}(1,0,c).
    $$
    This morphism is an isomorphism over $\mathbb{P}^1_X-\{0,\infty\}$, so the kernel and the cokernel induced morphism
    $$
    \psi_{N,m}^*\colon \pi_{X,*}F_{\mathcal{H}_\mathcal{X}^{(0)}(1,0,c)} \to \pi_{X,*}F_{\mathcal{H}_\mathcal{X}^{(N)}(Na,Nb,c)}
    $$
    are supported on $\{0,\infty\}\subset\mathbb{P}^1_X$.
    By \Cref{P1_vanishing_lemma}, this shows that the induced map
    $$
        \mathrm{H}^i\bigl(\mathbb{P}^1_X, \pi_{X,*}F_{\mathcal{H}_\mathcal{X}^{(0)}(1,0,c)}\bigr)\to
        \mathrm{H}^i\bigl(\mathbb{P}^1_X, \pi_{X,*}F_{\mathcal{H}_\mathcal{X}^{(N)}(Na,Nb,c)}\bigr)
    $$
    is surjective for $i>0$, so it suffices to prove that $\mathrm{H}^i\bigl(\mathbb{P}^1_X, \pi_{X,*}F_{\mathcal{H}_\mathcal{X}^{(0)}(1,0,c)}\bigr)=0$ for $i>0$.
    The scheme $H_X^{(0)}$ is isomorphic to $\mathbb{P}^1\times \mathbb{A}^1\times X$, and the modulus pair $\mathcal{H}_X^{(0)}(1,0,c)$ is isomorphic to $\bcube \otimes (\mathbb{A}^1,c[0])\otimes\mathcal{X}$.
    The morphism $\pi_X\colon H_X^{(0)}\to \mathbb{P}^1_X$ corresponds to $\pr_{13}\colon \mathbb{P}^1\times \mathbb{A}^1\times X\to \mathbb{P}^1\times X$.
    Therefore we have
    \begin{align*}
        &\mathrm{H}^i(\mathbb{P}^1_X, \pi_{X,*}F_{\mathcal{H}_\mathcal{X}^{(0)}(1,0,c)})\\
        \simeq{} &\mathrm{H}^i(\mathbb{P}^1\times X, \pr_{13,*}F_{\bcube \otimes (\mathbb{A}^1,c[0])\otimes\mathcal{X}})\\
        \simeq{} &\mathrm{H}^i(\mathbb{P}^1\times\mathbb{A}^1\times X, F_{\bcube \otimes (\mathbb{A}^1,c[0])\otimes\mathcal{X}})
        &\cdots\text{(affine vanishing property)}\\
        \simeq{} &\mathrm{H}^i(\mathbb{A}^1\times X, F_{(\mathbb{A}^1,c[0])\otimes\mathcal{X}})
        &\cdots\text{(cohomological cube invariance)}\\
        \simeq{} &\mathrm{H}^i(X, \pr_{2,*} F_{(\mathbb{A}^1,c[0])\otimes\mathcal{X}})
        &\cdots\text{(affine vanishing property)}\\
        ={}&0.
    \end{align*}
    This completes the proof.
\end{proof}

\begin{lemma}\label{Hirzebruch_vanishing_2}
	Suppose that $F\in \Sh_\Nis(\MCor^{\mathbb{Q}})$ has cohomological cube invariance, affine vanishing property, and left continuity.
	Let $\mathcal{X}$ be the henselian localization of an object of $\MCor^{\mathbb{Q},\ls}$ and $a,b\in \mathbb{Q}_{\geq 0}$, $a\neq 0$.
	Then we have
        $$
		\mathrm{H}^i(H^{(1)}_X, F_{\mathcal{H}^{(1)}_\mathcal{X}(a,b,a+b)})=0\quad (i>0).
	$$
\end{lemma}

\begin{proof}
    Let $M$ be a positive integer.
    If $M$ is large enough, then $N=2^M$ satisfies the assumption of \Cref{Hirzebruch_vanishing}.
    In this situation, \Cref{Hirzebruch_vanishing} shows that the group
    $$
        \mathrm{H}^i(H^{(1)}_X, F_{\mathcal{H}^{(1)}_\mathcal{X}(a,b,c)})=0
    $$
    is annihilated by $2^M$ for $i>0$, where $c=a+b-\frac{a+b+1}{2^M+1}$.
    Taking the colimit $M\to \infty$ and using the left continuity of $F$, we see that $\mathrm{H}^i(H^{(1)}_X, F_{\mathcal{H}^{(1)}_\mathcal{X}(a,b,c)})$ is a $2$-primary torsion for $i>0$.
    The same argument shows that $\mathrm{H}^i(H^{(1)}_X, F_{\mathcal{H}^{(1)}_\mathcal{X}(a,b,c)})$ is also a $3$-primary torsion for $i>0$, so this group vanishes.
\end{proof}

\begin{lemma}\label{blowup-a2}
    Suppose that $F\in \Sh_\Nis(\MCor^{\mathbb{Q}})$ has cohomological ls cube invariance, affine vanishing property, and left continuity.
    Let $\mathcal{X}\in \MCor^{\mathbb{Q},\ls}$ and let $\varphi\colon B_X\to \mathbb{A}^2_X$ denote the blow-up of $\mathbb{A}^2_X$ at the origin.
    Let $L=\mathbb{A}^1_X\times\{0\}$, $L'=\{0\}\times\mathbb{A}^1_X$ be two lines in $\mathbb{A}^2_X$.
    Then for any $a,b\in \mathbb{Q}_{\geq 0}$ with $a\neq 0$, we have
    $$
        \mathrm{R}^j\varphi_*F_{(B_X,\varphi^*(aL+bL'))}=0\quad(j>0).
    $$
\end{lemma}

\begin{proof}
    We may replace $\mathcal{X}$ by its henselian localization.
    Since $\varphi$ is an isomorphism on $\mathbb{A}^2_X-\{(0,0)\}$, the higher direct images $\mathrm{R}^j\varphi_*F_{(B_X,\varphi^*(aL+bL'))}\;(j>0)$ are supported on $\{(0,0)\}\subset\mathbb{A}^2_X$ and hence
	$$
		\mathrm{H}^i(\mathbb{A}^2_X, \mathrm{R}^j\varphi_*F_{(B_X,\varphi^*(aL+bL'))})=0\quad (i,j>0).
	$$
	On the other hand, we have
	\begin{align*}
		\mathrm{H}^i(\mathbb{A}^2_X, \varphi_*F_{(B_X,\varphi^*(aL+bL'))})
            ={}&\mathrm{H}^i(\mathbb{A}^2\times X,F_{(\mathbb{A}^2,a L+b L')\otimes \mathcal{X}})\\
            \simeq{}&\mathrm{H}^i(X,\pr_{2,*}F_{(\mathbb{A}^2,a L+b L')\otimes \mathcal{X}})=0\quad (i>0)
        \end{align*}
        by the affine vanishing property.
        Consider the Leray spectral sequence for $\varphi$:
        $$
            \mathrm{E}^{i,j}_2 = \mathrm{H}^i(\mathbb{A}^2_X, \mathrm{R}^j\varphi_*F_{(B_X,\varphi^*(aL+bL'))})
            \Rightarrow \mathrm{H}^{i+j}(B_X,F_{(B_X,\varphi^*(aL+bL'))}).
        $$
        We have $\mathrm{E}^{i,j}_2=0$ for $i>0$, so we get an isomorphism
	$$
		\mathrm{H}^0(\mathbb{A}^2_X, \mathrm{R}^j\varphi_*F_{(B_X,\varphi^*(aL+bL'))})\simeq \mathrm{H}^j(B_X,F_{(B_X,\varphi^*(aL+bL'))}).
	$$
	It suffices to show that the right hand side vanishes for $j>0$.
	Now recall that we have an isomorphism $B_X\simeq H^{(1)}_X$ which induces an isomorphism of modulus pairs $(B_X,\varphi^*(aL+bL'))\simeq \mathcal{H}_\mathcal{X}^{(1)}(a,b,a+b)$.
        Therefore the required vanishing follows from \Cref{Hirzebruch_vanishing}.
\end{proof}

The following lemma is an easy corollary of the universal property of blow-ups:

\begin{lemma}\label{blowup_exchange}
	Let $X$ be an algebraic $k$-scheme and $Z\subset Z'\subset X$ be closed subschemes.
	Let $Y$ denote the blow-up of $\Bl_ZX$ along the strict transform of $Z'$.
	Let $W$ denote the blow-up of $\Bl_{Z'}X$ along the inverse image of $Z$.
	Then there is an isomorphism $\rho\colon Y \xrightarrow{\sim} W$ over $X$.
\end{lemma}

\begin{proof}
	See \cite[Lemma 2.15]{BRS}.
\end{proof}

Now we prove our main theorem.

\begin{proof}[Proof of \Cref{main}]
	Let $\mathcal{X}\in \MCor^{\mathbb{Q},\ls}$ and let $Z\subset|D_X|$ be a smooth closed subscheme which has simple normal crossings with $|D_X|$.
    Let $\varphi\colon \Bl_Z\mathcal{X}\to \mathcal{X}$ denote the blow-up along $Z$.
    It suffices to show that the canonical morphism $F_\mathcal{X}\to \mathrm{R}\varphi_*F_{\Bl_Z\mathcal{X}}$ is an isomorphism.
    We proceed by induction on $r=\codim_X(Z)$.
    Since the claim is local in the Nisnevich topology, we may assume that $X=\Spec K\{x_1,\dots,x_d\}$, $|D_X|=\{x_1x_2\cdots x_s=0\}$, and $Z=\{x_{i_1}=\cdots=x_{i_r}=0\}$.
    Moreover, we may assume that $i_1=1$ since $Z$ is contained in $|D_X|$.
    The case $r=1$ is trivial; the blow-up along $Z$ is an isomorphism.
    If $r=2$, then the claim follows from \Cref{blowup-a2}.
    If $r\geq 3$, then we set $Z'=\{x_1=x_2=0\}\supset Z$.
    Let $\mathcal{Y}\to \Bl_Z\mathcal{X}$ denote the blow-up along the strict transform of $Z'$.
    Let $\mathcal{W}\to \Bl_{Z'}\mathcal{X}$ denote the blow-up along the inverse image of $Z$.
    Then there is an isomorphism $\rho\colon \mathcal{Y} \xrightarrow{\sim} \mathcal{W}$ over $\mathcal{X}$ by \Cref{blowup_exchange}.
    We have the following commutative diagram of $\mathbb{Q}$-modulus pairs:
    $$
    \xymatrix{
        \mathcal{Y}\ar[rr]^-{\rho}_-{\sim}\ar[d]^-{\varphi'}       &&\mathcal{W}\ar[d]^-{\psi'}\\
        \Bl_Z\mathcal{X}\ar[rd]_-\varphi          &&\Bl_{Z'}\mathcal{X}\ar[ld]^-\psi\\
        &\mathcal{X}.
    }
    $$
    Since the codimension of $Z'$ (resp. the strict transform of $Z'$) in $X$ (resp. in $\Bl_ZX$) is $2$, we have
    $$
        \mathrm{R}\psi_*F_{\Bl_{Z'}\mathcal{X}}\simeq F_\mathcal{X}\quad
        \text{(resp. }
        \mathrm{R}\varphi'_*F_\mathcal{Y}\simeq F_{\Bl_Z\mathcal{X}}
        \text{).}
    $$
    Moreover, since the codimension of $\psi^{-1}(Z)$ in $\Bl_{Z'}X$ is $r-1$, we have
    $\mathrm{R}\psi'_*F_{\mathcal{W}}\simeq F_{\Bl_{Z'}\mathcal{X}}$
    by the induction hypothesis.
    Combining these results, we obtain $\mathrm{R}\varphi_*F_{\Bl_Z\mathcal{X}}\simeq F_\mathcal{X}$.
\end{proof}

\section{Examples}

In this section, we construct two examples of sheaves on $\MCor^\mathbb{Q}$ and apply \Cref{main} to them.

\subsection{Differential forms}

Fix an integer $q\geq 0$.
Recall from \cite[Corollary A.6.1]{KSY1} that the sheaf of differential forms $\Omega^q=\bigl(X\mapsto \Gamma(X,\Omega^q_X)\bigr)$ can be regarded as an object of $\Sh_\Nis(\Cor)$.
Kelly-Miyazaki \cite[Theorem 4.3]{KellyMiyazaki_Hodge2} extended this to a sheaf $\MOmega^q$ on $\MCor^\ls$ in the case $\ch(k)=0$.
It is characterized by the fact that the formula
$$
	\MOmega^q(\mathcal{X})=\Gamma\bigl(X,\Omega^q_X(\log|D_X|)(D_X-|D_X|)\bigr)
$$
holds for $\mathcal{X}\in \MCor$.
In this subsection, we present a construction of $\MOmega^q$ using a ramification filtration.
This method is applicable also for $\ch(k)=p>0$, and it allows us to extend $\MOmega^q$ to $\MCor^{\mathbb{Q}}$.
We show that the resulting sheaf has cohomological cube invariance, affine vanishing property, and left continuity.
As a corollary of our main theorem, we obtain a new proof of the blow-up invariance of the Hodge cohomology with modulus due to Kelly-Miyazaki \cite[Theorem 6.1]{Kelly-Miyazaki}.

\begin{definition}
    Let $L\in \Phi$ and let $t\in \mathcal{O}_L$ be a uniformizer.
    We write $\Omega^q(\mathcal{O}_L)(\logrm)$ for the subgroup of $\Omega^q(L)$ generated by $\Omega^q(\mathcal{O}_L)$ and $\Omega^{q-1}(\mathcal{O}_L)\wedge\dlog(t)$.
    Since $\dlog(u)\in \Omega^1(\mathcal{O}_L)$ for $u\in \mathcal{O}_L^\times$, this does not depend on the choice of $t$.
    We define the \emph{logarithmic filtration} on $\Omega^q(L)$ by
    $$
    \Fil_r \Omega^q(L)=\begin{cases}
        \Omega^q(\mathcal{O}_L)&(r=0)\\
        t^{-\ceil{r}+1}\cdot\Omega^q(\mathcal{O}_L)(\logrm)&(r>0).
    \end{cases}
    $$
    This is also independent of the choice of $t$.
\end{definition}

\begin{lemma}\label{omega_max}
	Let $L\in \Phi$ and let $t\in \mathcal{O}_L$ be a uniformizer.
	For $\omega\in \Omega^q(\mathcal{O}_L)$, the following conditions are equivalent:
        \begin{enumerate}
            \item $\omega\in t\cdot \Omega^q(\mathcal{O}_L)(\logrm)$.
            \item $\omega\wedge \dlog(t)\in \Omega^{q+1}(\mathcal{O}_L)$.
        \end{enumerate}
\end{lemma}

\begin{proof}
    Suppose that (1) holds.
    We can write $\omega = t\omega_1+\omega_2\wedge\mathrm{d}t$ where $\omega_1\in \Omega^q(\mathcal{O}_L)$, $\omega_2\in \Omega^{q-1}(\mathcal{O}_L)$.
    It follows that $\omega\wedge\dlog(t) = \omega_1\wedge \mathrm{d}t\in \Omega^{q+1}(\mathcal{O}_L)$, so (1) implies (2).
    
    Next suppose that (2) holds.
    Writing $K$ for the residue field of $\mathcal{O}_L$, we have $\mathcal{O}_L\simeq K\{t\}$.
    Recall that we have an isomorphism
    $$
        \Omega^1(K[t])\simeq (K[t]\otimes_K \Omega^1(K))\oplus K[t]\mathrm{d}t.
    $$
    By henselization, we get an isomorphism $\Omega^1(\mathcal{O}_L)\simeq \bigl(\mathcal{O}_L\otimes_K\Omega^1(K)\bigr)\oplus \mathcal{O}_L dt$ and hence
    $$
        \Omega^q(\mathcal{O}_L)\simeq \bigl(\mathcal{O}_L\otimes_K\Omega^q(K)\bigr)\oplus \bigl(\mathcal{O}_L\otimes_K\Omega^{q-1}(K)\bigr)\wedge \mathrm{d}t.
    $$
    If we write $\omega = \alpha+\beta\wedge \mathrm{d}t$, we have $\omega\wedge\dlog(t) = t^{-1}\alpha\wedge \mathrm{d}t$.
    By (2), we get $\alpha\in t\cdot \Omega^q(K)$ and hence $\omega\in t\cdot \Omega^q(\mathcal{O}_L)(\logrm)$.
    Therefore (2) implies (1).
\end{proof}

\begin{lemma}\label{omega_trace}
	Let $L\in \Phi$ and let $t\in \mathcal{O}_L$ be a uniformizer.
	Let $L'/L$ be a finite extension and let $t'\in \mathcal{O}_{L'}$ be a uniformizer.
	Then we have $\Tr_{L'/L}\bigl(t'\cdot \Omega(\mathcal{O}_{L'})(\logrm)\bigr)\subset t\cdot \Omega(\mathcal{O}_L)(\logrm)$.
\end{lemma}

\begin{proof}
	Write $t=u(t')^e$ with $u\in \mathcal{O}_{L'}^\times$.
	For any $\omega\in \Omega^q(\mathcal{O}_{L'})(\logrm)$, we have
	\begin{align*}
		\Tr_{L'/L}(t' \omega)\wedge \dlog(t)
            &{}=\Tr_{L'/L}\bigl(t' \omega\wedge\dlog(u(t')^e)\bigr)\\
		&{}=\Tr_{L'/L}\bigl(t' \omega\wedge \dlog(u)\bigr)
            + e\cdot\Tr_{L'/L}\bigl(t' \omega\wedge \dlog(t')\bigr).
        \end{align*}
        We can write $\omega = \omega_1+\omega_2\wedge\dlog(t')$ with $\omega_1\in \Omega^q(\mathcal{O}_{L'})$, $\omega_2\in \Omega^{q-1}(\mathcal{O}_{L'})$.
        Then we have $t'\omega\wedge\dlog(t') = \omega_1\wedge \mathrm{d}t\in \Omega^{q+1}(\mathcal{O}_{L'})$, so we see that
        $$
		\Tr_{L'/L}(t' \omega)\wedge \dlog(t)\in \Tr_{L'/L}\bigl(\Omega^{q+1}(\mathcal{O}_{L'})\bigr)\subset \Omega^{q+1}(\mathcal{O}_L).
	$$
	By \Cref{omega_max}, this implies that $\Tr_{L'/L}(t' \omega)\in t\cdot \Omega^q(\mathcal{O}_L)(\logrm)$.
\end{proof}

\begin{lemma}\label{omega_trace_axiom}
	Let $L\in \Phi$ and let $L'/L$ be a finite extension of ramification index $e$.
	Then we have $\Tr_{L'/L}\bigl(\Fil_r\Omega^q(L')\bigr)\subset \Fil_{r/e}\Omega^q(L)$ for any $r\in \mathbb{Q}_{\geq 0}$.
    In other words, $\{\Fil_r\Omega^q(L)\}_{r\in \mathbb{Q}_{\geq 0},L\in \Phi}$ defines a ramification filtration on $\Omega^q$.
\end{lemma}

\begin{proof}
    We may assume that $r>0$.
    Take uniformizers $t\in \mathcal{O}_L$, $t'\in \mathcal{O}_{L'}$ and write $t=u(t')^e$ with $u\in \mathcal{O}_{L'}^\times$.
    Let $\omega\in \Fil_r \Omega^q(L')$.
    Then we have $t'^{\ceil{r}}\omega\in t'\cdot \Omega^q(\mathcal{O}_{L'})(\logrm)$.
    Since $e\ceil{r/e}\geq \ceil{r}$, this implies $t'^{e\ceil{r/e}}\omega\in t'\cdot \Omega^q(\mathcal{O}_{L'})(\logrm)$.
    Therefore we have
	\begin{align*}
		t^{\ceil{r/e}}\Tr_{L'/L}(\omega)&{}=\Tr_{L'/L}(u^{\ceil{r/e}}t'^{e\ceil{r/e}}\omega)\\
		&{}\in \Tr_{L'/L}\bigl(t'\cdot \Omega^q(\mathcal{O}_{L'})(\logrm)\bigr)\\
            &{}\subset t\cdot \Omega^q(\mathcal{O}_L)(\logrm).
	\end{align*}
    Here, we used \Cref{omega_trace} for the last inclusion.
    This shows that $\Tr_{L'/L}(\omega)\in \Fil_{r/e}\Omega^q(L)$.
\end{proof}

We write $\MOmega^q$ for the Nisnevich sheaf on $\MCor^\mathbb{Q}$ associated to the logarithmic filtration on $\Omega^q$.
The next lemma shows that this coincides with Kelly-Miyazaki's definition for $\ch(k)=0$:

\begin{lemma}\label{omega_log_formula}
	For any $\mathcal{X}\in \MCor^{\mathbb{Q},\ls}$, we have
	$$
		\MOmega^q(\mathcal{X})=\Gamma\bigl(X,\Omega^q_X(\log |D_X|)(\ceil{D_X}-|D_X|)\bigr).
	$$
	In particular, $\MOmega^q$ has left continuity and affine vanishing property.
\end{lemma}

\begin{proof}
        It is clear that the left hand side is contained in the right hand side.
        Let us prove the reverse inclusion.
        We may assume that $\mathcal{X}$ has a coordinate $x_1,\dots,x_d$ such that $D_X=\sum_{i=1}^s r_i D_i$ where $D_i=\{x_i=0\}$.
        Let $\omega\in \Gamma\bigl(X,\Omega^q_X(\log |D_X|)(\ceil{D_X}-|D_X|)\bigr)$.
        Let $L\in \Phi$ and let $t\in \mathcal{O}_L$ be a uniformizer.
        Suppose that we have a commutative diagram of the form \Cref{DVR_diagram}.
	We set $e_i=v_L(\rho^*x_i)$ and assume that $e_i>0$ for some $i$.
	Then there is some unit $u\in \mathcal{O}_L^\times$ such that
	$$
	\rho^*(x_1^{\ceil{r_1}-1}\cdots x_s^{\ceil{r_s}-1}\omega)
        = u t^{\sum_{i=1}^s e_i(\ceil{r_i}-1)}\rho^*\omega.
        $$
        The left hand side is an element of $\Omega^q(\mathcal{O}_L)(\logrm)$, so we get
        $$
            t^{\sum_{i=1}^s e_i(\ceil{r_i}-1)}\rho^*\omega \in \Omega^q(\mathcal{O}_L)(\logrm).
        $$
	Now we observe that the inequality $\sum_{i=1}^s e_i(\ceil{r_i}-1)<\sum_{i=1}^s e_i r_i\leq \ceil{\sum_{i=1}^s e_i r_i}$ implies $\sum_{i=1}^s e_i(\ceil{r_i}-1)\leq \ceil{\sum_{i=1}^s e_i r_i}-1$.
        Therefore we get
	$$
		t^{\ceil{\sum_{i=1}^s e_i r_i}-1}\rho^*\omega\in \Omega^q(\mathcal{O}_L)(\logrm).
	$$
	This shows that $\omega\in \MOmega^q(\mathcal{X})$ and hence the equality holds.
    The left continuity is clear from the equality, and the affine vanishing property follows from the fact that the sheaf $\Omega^q_X(\log |D_X|)(\ceil{D_X}-|D_X|)$ is a quasi-coherent $\mathcal{O}_X$-module.
\end{proof}

\begin{lemma}\label{omega_CCI}
	The modulus sheaf $\MOmega^q$ has cohomological cube invariance.
\end{lemma}

\begin{proof}
	This follows from \Cref{omega_log_formula} and \cite[Corollary 5.2]{KellyMiyazaki_Hodge2}.
	We sketch the proof for the convenience of the reader.
	Let $\mathcal{X}\in \MCor^{\mathbb{Q},\ls}$.
	An explicit computation shows that, as a quasi-coherent sheaf on $X\times\mathbb{P}^1$, we have
	\begin{align*}
		(\MOmega^q)_{\mathcal{X}\otimes\bcube}&{}\simeq \pr_1^*(\MOmega^q)_\mathcal{X}\oplus\bigl(\pr_1^*(\MOmega^{q-1})_\mathcal{X}\otimes\pr_2^*(\MOmega^1)_\bcube\bigr)\\
		&{}\simeq \pr_1^*(\MOmega^q)_\mathcal{X}\oplus\bigl(\pr_1^*(\MOmega^{q-1})_\mathcal{X}\otimes\pr_2^*\mathcal{O}_{\mathbb{P}^1}(-1)\bigr)
	\end{align*}
	Therefore it suffices to show that for any finite locally free $\mathcal{O}_X$-modules $E,E'$ on $X$, the morphism
        $$
            \RGamma (X,E) \to \RGamma\bigl(X\times\mathbb{P}^1, \pr_1^*E\oplus(\pr_1^*E'\otimes \pr_2^*\mathcal{O}_{\mathbb{P}^1}(-1))\bigr)
        $$
        is an isomorphism.
	This can be reduced to the case $E=E'=\mathcal{O}_X$, which is easy.
\end{proof}

\begin{corollary}[{\cite[Theorem 6.1]{KellyMiyazaki_Hodge2}}]\label{Omega_CBI}
    Let $\mathcal{X}\in \MCor^{\mathbb{Q},\ls}$ and let $Z\subset |D_X|$ be a smooth closed subscheme which has simple normal crossings with $|D_X|$.
    Then the morphism
    $$
        \RGamma(X,\MOmega^q_\mathcal{X})\to \RGamma(\Bl_ZX,\MOmega^q_{\Bl_Z\mathcal{X}})
    $$
    is an isomorphism.
\end{corollary}

\begin{proof}
    Since $\MOmega^q$ has cohomological cube invariance by \Cref{omega_CCI} and has left continuity and affine vanishing property by \Cref{omega_log_formula}, the claim follows from \Cref{main}.
\end{proof}

\subsection{Witt vectors}
In this subsection we assume that $\ch(k)=p>0$.
Fix an integer $n>0$.
The sheaf of $p$-typical Witt vectors $\mathrm{W}_n=\bigl(X\mapsto \Gamma(X,\mathrm{W}_n\mathcal{O}_X)\bigr)$ is representable by an algebraic group, so it can be regarded as an object of $\Sh_\Nis(\Cor)$ \cite[Lemma 1.4.4]{BVK16}.
In this subsection, we construct a modulus sheaf $\MW_n\in \Sh_\Nis(\MCor^{\mathbb{Q}})$ extending $\mathrm{W}_n$ using the Brylinski-Kato filtration.
We prove that $\mathrm{W}_n$ has cohomological cube invariance, affine vanishing property, and left continuity.
As a corollary of our main theorem, we obtain the blow-up invariance of the Witt vector cohomology with modulus.

\begin{definition}
	Let $L\in \Phi$ and let $t\in \mathcal{O}_L$ be a uniformizer.
	We define the \emph{Brylinski-Kato filtration} \cite[Definition 3.1]{Kato89} on $\mathrm{W}_n(L)$ by
	$$
		\Fil_r \mathrm{W}_n(L)=\begin{cases}
			\mathrm{W}_n(\mathcal{O}_L)&(r=0)\\
			\bigl\{a\in \mathrm{W}_n(L)\bigm| [t^{\ceil{r}-1}]\mathrm{F}^{n-1}(a)\in \mathrm{W}_n(\mathcal{O}_L)\bigr\}&(r>0).
		\end{cases}
	$$
        Here, $[a]$ means $(a,0,\dots,0)$ and $\mathrm{F}$ denotes the Frobenius map $(a_0,\dots,a_{n-1})\mapsto (a_0^p,\dots,a_{n-1}^p)$.
	This definition does not depend on the choice of the uniformizer $t$.
\end{definition}

\begin{remark}\label{Witt_filtration_valuation}
	For $a=(a_0,\dots,a_{n-1})\in \mathrm{W}_n(L)$, we have
	$$
		[t^{\ceil{r}-1}]\mathrm{F}^{n-1}(a)=(t^{\ceil{r}-1}a_0^{p^{n-1}},t^{p(\ceil{r}-1)}a_1^{p^{n-1}},\dots,t^{p^{n-1}(\ceil{r}-1)}a_{n-1}^{p^{n-1}}).
	$$
	Therefore we have $a\in \Fil_r \mathrm{W}_n(L)$ if and only if
	$$
		p^{n-1}v_L(a_i) +  p^i (\ceil{r}-1) \geq 0
	$$
	holds for $i=0,1,\dots,n-1$.
	In particular, we have $\Fil_1\mathrm{W}_n(L)=\Fil_0\mathrm{W}_n(L)=\mathrm{W}_n(\mathcal{O}_L)$.
\end{remark}

\begin{lemma}\label{Witt_trace}
	Let $L\in \Phi$.
	Let $L'/L$ be a finite extension and $t'\in \mathcal{O}_{L'}$ be a uniformizer.
	If $\overline{L}$ is an algebraic closure of $L$, then we have
	$$i_{\overline{L}/L}\circ \Tr_{L'/L}\bigl([t']\cdot \mathrm{W}_n(\mathcal{O}_{L'})\bigr)\subset [t']\cdot \mathrm{W}_n(\mathcal{O}_{\overline{L}})$$
	where $i_{\overline{L}/L}\colon \mathrm{W}_n(\mathcal{O}_L)\to \mathrm{W}_n(\mathcal{O}_{\overline{L}})$ is the inclusion.
\end{lemma}

\begin{proof}
	Let $a\in \mathrm{W}_n(\mathcal{O}_{L'})$.
	If $m$ (resp. $n$) is the separable (resp. inseparable) degree of $L'/L$, then there are $m$ distinct $L$-embeddings $\sigma_1,\dots,\sigma_m\colon L'\to \overline{L}$ and we have
	$$
		i_{\overline{L}/L}\circ \Tr_{L'/L}([t']\cdot a)=n\cdot \sum_{i=1}^m\sigma_i([t']\cdot a).
	$$
	It is clear that the right hand side is contained in $[t']\cdot \mathrm{W}_n(\mathcal{O}_{\overline{L}})$.
\end{proof}

\begin{lemma}\label{Witt_trace_axiom}
	Let $L\in \Phi$ and let $L'/L$ be a finite extension of ramification index $e$.
	Then we have $\Tr_{L'/L}(\Fil_r \mathrm{W}_n(L'))\subset \Fil_{r/e}\mathrm{W}_n(L)$ for any $r\in \mathbb{Q}_{\geq 0}$.
\end{lemma}

\begin{proof}
	Take uniformizers $t\in \mathcal{O}_L$, $t'\in \mathcal{O}_{L'}$ and write $t=u(t')^e$ with $u\in \mathcal{O}_{L'}^\times$.
	Let $a\in \Fil_r \mathrm{W}_n(L')$.
	Then we have $[t'^{\ceil{r}-1}]\mathrm{F}^{n-1}(a)\in \mathrm{W}_n(\mathcal{O}_{L'})$.
	Since $e\ceil{r/e}\geq \ceil{r}$, this implies $[t'^{e\ceil{r/e}}]\mathrm{F}^{n-1}(a)\in [t']\cdot \mathrm{W}_n(\mathcal{O}_{L'})$.
        Therefore we have
	\begin{align*}
		i_{\overline{L}/L}\bigl([t^{\ceil{r/e}}]\mathrm{F}^{n-1}(\Tr_{L'/L}(a))\bigr)
		&{}=i_{\overline{L}/L}\circ\Tr_{L'/L}\bigl([u^{\ceil{r/e}}t'^{e\ceil{r/e}}]\mathrm{F}^{n-1}(a)\bigr)\\
		&{}\in i_{\overline{L}/L}\circ \Tr_{L'/L}\bigl([t']\cdot \mathrm{W}_n(\mathcal{O}_{L'})\bigl)\\
            &{}\subset [t']\cdot \mathrm{W}_n(\mathcal{O}_{\overline{L}}).
	\end{align*}
        Here, we used \Cref{Witt_trace} for the last inclusion.
	If we set $\Tr_{L'/L}(a)=b=(b_0,\dots,b_{n-1})$, then the above computation shows that
	$$
		p^{n-1}v_L(b_i) + p^i \ceil{r/e}\geq p^i/e
	$$
	for $i=0,1,\dots,n-1$.
	Since the left hand side belongs to $p^i\mathbb{Z}$, we can replace the right hand side by $p^i$.
        Therefore we get
	$$
		p^{n-1}v_L(b_i)+p^i(\ceil{r/e}-1)\geq 0
	$$
	for $i=0,1,\dots,n-1$.
	This implies $\Tr_{L'/L}(a)=b\in \Fil_{r/e}\mathrm{W}_n(L)$.
\end{proof}

We write $\MW_n$ for the Nisnevich sheaf on $\MCor^\mathbb{Q}$ associated to the Brylinski-Kato filtration on $\mathrm{W}_n$.
In order to give an explicit formula for $\MW_n$, we recall the notion of \emph{Witt divisorial sheaf} which was introduced by Tanaka \cite{Tanaka_Witt}.

\begin{definition}[{\cite[Definition 3.1]{Tanaka_Witt}}]
	Let $X\in \Sm$ and let $D=\sum_{i=1}^s r_i D_i$ be a $\mathbb{Q}$-divisor on $X$.
	Let $\xi_i$ denote the generic point of $D_i$ and $v_i$ denote the discrete valuation on $\mathcal{O}_{X,\xi_i}$.
	Let $j\colon X\setminus|D|\to X$ denote the inclusion.
	The \emph{Witt divisorial sheaf} $\mathrm{W}_n\mathcal{O}_X(D)$ is the Zariski subsheaf of $j_*\mathrm{W}_n\mathcal{O}_{X\setminus |D|}$ defined by
	$$
		U\mapsto \bigl\{(a_0,\dots,a_{n-1})\in \Gamma(U, j_*\mathrm{W}_n\mathcal{O}_{X\setminus|D|})\bigm| \xi_i\in U\implies v_i(a_i)+p^i r_i\geq 0\bigr\}.
	$$
	It is known that $\mathrm{W}_n\mathcal{O}_X(D)$ is a quasi-coherent $\mathrm{W}_n\mathcal{O}_X$-module \cite[Lemma 3.5 (2)]{Tanaka_Witt}.
\end{definition}

\begin{lemma}\label{Witt_log_formula}
	For any $\mathcal{X}\in \MCor^{\mathbb{Q},\ls}$, we have
	$$
		\MW_n(\mathcal{X})=\Gamma\bigl(X,\mathrm{W}_n\mathcal{O}_X((\ceil{D_X}-|D_X|)/p^{n-1})\bigr).
	$$
	In particular, $\MW_n$ has left continuity and affine vanishing property.
\end{lemma}

\begin{proof}
	It is clear from \Cref{Witt_filtration_valuation} that the left hand side is contained in the right hand side.
	Let us prove the reverse inclusion.
	We may assume that $\mathcal{X}$ has a coordinate $x_1,\dots,x_d$ such that $D_X=\sum_{i=1}^s r_i D_i$ where $D_i=\{x_i=0\}$.
	Let $a\in \Gamma\bigl(X,\mathrm{W}_n\mathcal{O}_X((\ceil{D_X}-|D_X|)/p^{n-1})\bigr)$.
	Let $L\in \Phi$ and let $t\in \mathcal{O}_L$ be a uniformizer.
	Suppose that we have a commutative diagram of the form \Cref{DVR_diagram}.
	We set $e_i=v_L(\rho^*x_i)$ and assume that $e_i>0$ for some $i$.
	Then there is some unit $u\in \mathcal{O}_L^\times$ such that
	$$
		\rho^*\bigl([x_1^{\ceil{r_1}-1}\cdots x_s^{\ceil{r_s}-1}]\mathrm{F}^{n-1}(a)\bigr)
            =
            [u t^{\sum_{i=1}^s e_i(\ceil{r_i}-1)}]\mathrm{F}^{n-1}(\rho^*a).
        $$
        The left hand side is an element of
        $\mathrm{W}_n(\mathcal{O}_L)$, so we get
        $$
            [t^{\sum_{i=1}^s e_i (\ceil{r_i}-1)}]\mathrm{F}^{n-1}(\rho^*a) \in \mathrm{W}_n(\mathcal{O}_L).
        $$
	Now we observe that the inequality $\sum_{i=1}^s e_i(\ceil{r_i}-1)<\sum_{i=1}^s e_i r_i\leq \ceil{\sum_{i=1}^s e_i r_i}$ implies $\sum_{i=1}^s e_i(\ceil{r_i}-1)\leq \ceil{\sum_{i=1}^s e_i r_i}-1$.
        Therefore we get
	$$
		[t^{\ceil{\sum_{i=1}^s e_i r_i}-1}]\mathrm{F}^{n-1}(\rho^*a) \in \mathrm{W}_n(\mathcal{O}_L).
	$$
    This shows that $a\in \MW_n(\mathcal{X})$ and hence the equality holds.
    The left continuity is clear from the equality, and the affine vanishing property follows from the fact that the sheaf $\mathrm{W}_n\mathcal{O}_X((\ceil{D_X}-|D_X|)/p^{n-1})$ is a quasi-coherent $\mathrm{W}_n\mathcal{O}_X$-module.
\end{proof}

\begin{lemma}\label{Witt_CCI}
	The modulus sheaf $\MW_n$ has cohomological ls cube invariance.
\end{lemma}

\begin{proof}
    Let $X\in \Sm$.
    By \Cref{Witt_log_formula}, it suffices to show that the canonical morphism
    $$
        \RGamma\bigl(X,\mathrm{W}_n\mathcal{O}_X(D)\bigr)\to \RGamma\bigl(X\times\mathbb{P}^1,\mathrm{W}_n\mathcal{O}_{X\times\mathbb{P}^1}(\pr_1^*D)\bigr)
    $$
    is an isomorphism for any $\mathbb{Q}$-divisor $D$ on $X$.
    We prove this by induction on $n$.
    If $n=1$, then what we have to show is that the canonical morphism
    $\mathcal{O}_X(D)\to \mathrm{R}\pr_{1,*}\pr_1^*\mathcal{O}_X(D)$
    is an isomorphism, which is easy.
    Suppose that $n\geq 2$.
    By \cite[Proposition 3.7]{Tanaka_Witt}, we have an exact sequence
    $$
        0\to (\mathrm{F}^{n-1}_X)_*\mathcal{O}_X(p^{n-1}D)\to \mathrm{W}_n\mathcal{O}_X(D)\to \mathrm{W}_{n-1}\mathcal{O}_X(D)\to 0
    $$
    of Zariski sheaves on $X$, where $\mathrm{F}_X\colon X\to X$ is the absolute Frobenius morphism.
    We also have a similar exact sequence of Zariski sheaves on $X\times\mathbb{P}^1$:
    $$
        0\to (\mathrm{F}^{n-1}_{X\times \mathbb{P}^1})_*\mathcal{O}_{X\times \mathbb{P}^1}(p^{n-1}\pr_1^*D)\to \mathrm{W}_n\mathcal{O}_{X\times \mathbb{P}^1}(\pr_1^*D)\to \mathrm{W}_{n-1}\mathcal{O}_{X\times \mathbb{P}^1}(\pr_1^*D)\to 0.
    $$
    Therefore the claim follows from the induction hypothesis and the five lemma.
\end{proof}

\begin{proof}[Proof of \Cref{Witt_CBI}]
    Since $\MW_n$ has cohomological cube invariance by \Cref{Witt_CCI} and has left continuity and affine vanishing property by \Cref{Witt_log_formula}, the claim follows from \Cref{main}.
\end{proof}

\section{Relation to motives with modulus}

In this section we discuss the relationship between our main theorem and the theory of motives with modulus.
We show that the Hodge cohomology with modulus and the Witt vector cohomology with modulus are representable in the category of motives with modulus if we assume resolution of singularities.
\begin{definition}\label{def:resolution}
	We say that $k$ admits \emph{resolution of singularities} if the following conditions hold:
	\begin{enumerate}
            \item[(RS1)]
            For any $\mathcal{X}\in \MCor$, there exists $\mathcal{Y}\in \MCor^\ls$ and a proper minimal ambient morphism $f\colon \mathcal{Y}\to \mathcal{X}$ such that $f^\circ\colon Y^\circ\to X^\circ$ is an isomorphism.
            \item[(RS2)]
            Let $\mathcal{X}\in \MCor^\ls$ and let $f\colon \mathcal{Y}\to \mathcal{X}$ be a proper minimal ambient morphism such that $f^\circ\colon Y^\circ\to X^\circ$ is an isomorphism.
            Then there is a sequence of blow-ups $\mathcal{X}_n\to \mathcal{X}_{n-1}\to\dots\to \mathcal{X}_0=\mathcal{X}$ with the following properties:
            \begin{enumerate}
                \item The center of the blow-up $\mathcal{X}_i\to \mathcal{X}_{i-1}$ is a smooth closed subscheme of $|D_X|$ which has simple normal crossings with $|D_X|$.
                \item The morphism $\mathcal{X}_n\to \mathcal{X}$ factors through $\mathcal{Y}$.
            \end{enumerate}
	\end{enumerate}
\end{definition}

We recall the construction of the category of motives with modulus.
For $\mathcal{X}\in \MCor$, we write $\mathbb{Z}_\tr(\mathcal{X})$ for the presheaf of abelian groups on $\MCor$ represented by $\mathcal{X}$.
It is known that $\mathbb{Z}_\tr(\mathcal{X})$ is a Nisnevich sheaf \cite[Proposition 3.3.5]{KMSY1}.

\begin{definition}
    Let $\mathcal{D}\bigl(\Sh_\Nis(\MCor)\bigr)$ denote the derived $\infty$-category of $\Sh_\Nis(\MCor)$.
    Let $(\mathrm{CI})$ be the localizing subcategory of $\mathcal{D}\bigl(\Sh_\Nis(\MCor)\bigr)$ generated by the cofibers of $\pr_1\colon \mathbb{Z}_\tr(\mathcal{X}\otimes \bcube)\to \mathbb{Z}_\tr(\mathcal{X})$ for $\mathcal{X}\in \MCor$.
	The category of motives with modulus $\MDM^\eff$ is defined as the Verdier quotient
	$$
		\MDM^\eff = \frac{\mathcal{D}\bigl(\Sh_\Nis(\MCor)\bigr)}{(\mathrm{CI})}
	$$
	in the sense of \cite[Definition 5.4]{BGT}.
	By \cite[Theorem 4.1.1]{KMSY3}, this is equivalent to the original definition \cite[Definition 3.2.4]{KMSY3} after taking the homotopy category.
    We write $\mathrm{M}(\mathcal{X})$ for the image of $\mathbb{Z}_\tr(\mathcal{X})$ in $\MDM^\eff$.
\end{definition}

\begin{lemma}\label{general_realization}
    Assume that $k$ admits resolution of singularities.
    Let $F$ be a Nisnevich sheaf on $\MCor$ which has cohomological cube invariance.
    Suppose that for any $\mathcal{X}\in \MCor^\ls$ and any smooth closed subscheme $Z\subset|D_X|$ which has simple normal crossings with $|D_X|$, the morphism
    $$
         \RGamma(X,F_\mathcal{X})\to \RGamma(\Bl_ZX,F_{\Bl_Z\mathcal{X}})
    $$
    is an isomorphism.
    Then there is an object $\mathbf{F}\in \MDM^\eff$ such that for $\mathcal{X}\in \MCor^\ls$, there is an isomorphism
	$$
		\map_{\MDM^\eff}\bigl(\mathrm{M}(\mathcal{X}),\mathbf{F}\bigr) \simeq \RGamma(X,F_\mathcal{X}).
	$$
\end{lemma}

\begin{proof}
    We define $\mathbf{F}$ to be the image of $F$ in $\MDM^\eff$.
    We prove that this object has the required property.
    Let $\mathcal{X}\in \MCor^\ls$.
	By \cite[Theorem 4.6.3]{KMSY1} we have
	\begin{align*}
		\Ext^i_{\Sh_\Nis(\MCor)}\bigl(\mathbb{Z}_\tr(\mathcal{X}),F\bigr)\simeq \colim_{\mathcal{Y}\to \mathcal{X}} \mathrm{H}^i(Y,F_\mathcal{Y})
	\end{align*}
	where the colimit is taken over all proper minimal ambient morphisms $\mathcal{Y}\to \mathcal{X}$ such that $f^\circ\colon Y^\circ\to X^\circ$ is an isomorphism.
	Since $k$ is assumed to satisfy (RS2), it follows from our assumption that the last term is isomorphic to $\mathrm{H}^i(X,F_\mathcal{X})$.
	Therefore the canonical morphism
	$$
		\mathrm{R}\Gamma(X,F_\mathcal{X})\to \map_{\mathcal{D}(\Sh_\Nis(\MCor))} \bigl(\mathbb{Z}_\tr(\mathcal{X}),F\bigr)
	$$
	is an equivalence.
    To show that the right hand side is isomorphic to $\map_{\MDM^\eff}\bigl(\mathrm{M}(\mathcal{X}),\mathbf{F}\bigr)$, it suffices to show that $F$ is cube-local, that is, the morphism
	$$
		\map_{\mathcal{D}(\Sh_\Nis(\MCor))} \bigl(\mathbb{Z}_\tr(\mathcal{Y}),F\bigr)\to
		\map_{\mathcal{D}(\Sh_\Nis(\MCor))} \bigl(\mathbb{Z}_\tr(\mathcal{Y}\otimes\bcube),F\bigr)
	$$
	is an equivalence for any $\mathcal{Y}\in \MCor$.
	By (RS1), we may assume that $\mathcal{Y}\in \MCor^\ls$.
	Then this morphism can be identified with $\mathrm{R}\Gamma(Y,F_\mathcal{Y})\to \mathrm{R}\Gamma(Y\times\mathbb{P}^1,F_\mathcal{Y\otimes\bcube})$, so the claim follows from the cohomological cube invariance.
\end{proof}

\begin{proof}[Proof of \Cref{cor:realization}]
    The sheaf $\MOmega^q$ (resp. $\MW_n$) has cohomological cube invariance by \Cref{omega_CCI} (resp. \Cref{Witt_CCI}).
    It also satisies the condition of \Cref{general_realization} by \Cref{Omega_CBI} (resp. \Cref{Witt_CBI}).
    Therefore the claim follows from \Cref{general_realization}.
\end{proof}

\printbibliography

\end{document}